\documentclass[12pt]{article}
\usepackage{hyperref}
\pdfoutput=1 

\usepackage{setspace}

\usepackage{delbid}
\usepackage{enumitem}
\usepackage{graphicx}
\setlist{itemsep=0em,topsep=.3em}
\usepackage{afterpage}

\begin{document}

\title{Bias and dessins}
\author{Jonathan Fine}
\date{21 June 2015 (with Postscript 14 February 2018)}

\maketitle


\section{Introduction}

\aaa[Abstract.]
Grothendieck's theory of dessins provides a bridge between algebraic
numbers and combinatorics.  This paper adds a new concept, called
\emph{bias}, to the bridge.  This produces: (i)~from a biased plane
tree the construction of a sequence of algebraic numbers, and (ii)~a
Galois invariant lattice structure on the set of biased dessins.  Bias
brings these benefits by (i)~using individual polynomials instead of
equivalence classes of polynomials, and (ii)~applying properties
of covering spaces and the fundamental group.  The new features give
new opportunities.

At the 2014 SIGMAP conference the author spoke~\cite{jfine-dlbd-talk}
on \emph{The decorated lattice of biased dessins}.  This decorated
lattice $\mathcal{L}$ is combinatorially defined, and its automorphism
group contains the absolute Galois group $\Gamma$, perhaps as an
index~$6$ subgroup.

This paper defines new families of invariants of dessins, although
they require further work to be understood and useful.  For this,
$\mathcal{L}$ is vital.  This paper relies on the the existing,
unbiased, theory.  Also, it only sketches the construction of
$\mathcal{L}$.  In \cite{jfine-rpt-an, jfine-dlbd} the author will
remove this dependency, develop the biased theory further, with a
focus on $\Gamma$, and make the theory more accessible.  [[\emph{The
    Postscript (page~15) should be read next. This paper is otherwise
    unchanged.}]]

\aaa[Advice to the reader.]
This paper is a compromise.  Either directly or in the background it
involves algebraic numbers, algebraic geometry, analysis,
combinatorics, Galois theory and topology.  What assumptions to make
of the reader?  For example, the Galois invariance of the lattice
structure (Theorem~\ref{thm-ls-gi}) will be obvious to some readers,
and mysterious to others. The paper assumes only what is required to
achieve its limited goal.

This goal is to show that the addition of bias greatly improves the
existing theory of dessins.  Central to dessins is the bijection given
by the bridge between algebraic numbers and combinatorics.
Theorem~\ref{thm-bt-bij} gives a bridge which carries bias.  Given the
stated analogous unbiased result, its proof should be accessible to
all readers.  This gives (see~\S\ref{s-sp-gi}) many new Galois
invariants for biased Shabat polynomials.

The join operation on biased dessins is new.  Its combinatorial
description (see Definition~\ref{dfn-lat-join}) is simple and
attractive.  It is also Galois invariant, which we prove
elsewhere~\cite{jfine-dlbd}.  It gives a powerful method
(see~\S\ref{s-cnc-tgi}) of producing new Galois invariants of biased
dessins from old. In \S\ref{s-bd-cs-pi} to \S\ref{s-bd-dl} we motivate
and sketch the definition of the decorated lattice $\mathcal{L}$ of
biased dessins.  This provides the ground for the definition of
further new Galois invariants of dessins.

Even when studying unbiased objects, use of bias is a great help (see
\S\ref{s-cnc-ubgi}).  The author will in~\cite{jfine-rpt-an,jfine-dlbd}
present the theory of dessins anew, but with bias introduced from the
very beginning, rather than as an afterthought (the present paper's
approach).  Further, the focus will be on the absolute Galois group,
and on making the theory more widely accessible.  Until then, there is
the present paper, with its limited goal.

In the rest of this section we give the basic concepts on which most
of this paper is based.  Sections \S\ref{s-bd-cs-pi}--\ref{s-bd-dl}
require further background.

\aaa[\label{s-i-an}Algebraic numbers.]
We let $\mathbb{Q}\subset\mathbb{C}$ denote the rational and complex
numbers.  Let $f(z)$ be a polynomial in $z$, with rational
coefficients.  If $f(u) = 0$ we say that $u$ is an \emph{algebraic
  number} (provided $u\in\mathbb{C}$ and $f$ is not constant).  The
algebraic numbers form a field, $\Qbar$, lying between $\mathbb{Q}$
and $\mathbb{C}$.

We let $\mathbb{Q}[z] \subset \Qbar[z] \subset \mathbb{C}[z]$ denote
polynomials with coefficients in
$\mathbb{Q}\subset\Qbar\subset\mathbb{C}$ respectively. By the
fundamental theorem of algebra (a topological result proved by Gauss),
the field $\mathbb{C}$ is \emph{algebraically closed}.  In other
words, any $f\in\mathbb{C}[z]$ has $n=\deg f$ roots, when counted with
multiplicity.  The same is true of $\Qbar$, but this is an algebraic
result.

\begin{definition}
The \defemph{absolute Galois group} $\Gamma$ consists of all field
automorphisms of\/ $\Qbar$.
\end{definition}

We need some simple results about $\Gamma$.  We use
$u\mapsto\tilde{u}$ to denote an element of $\Gamma$. Always,
$\tilde{u}=u$ for $u\in\mathbb{Q}$.  By acting on coefficients this
induces a map $f\mapsto\tilde{f}$ on $\Qbar[z]$.  Because
$u\mapsto\tilde{u}$ is a field automorphism, it follows that
$\tilde{f}(\tilde{u}) = \tilde{v}$, where $v=f(u)$.  Similarly, for
derivatives.  The expression $\tilde{f}'$ can be evaluated in two
ways: first apply $u\mapsto\tilde{u}$ and then the derivative, or vice
versa.  Both give the same result, which we denote by $\tilde{f}'$.

The inclusion $\Qbar\subset\mathbb{C}$ induces a topology on $\Qbar$.
Note that $u\mapsto\tilde{u}$ in $\Gamma$ is not continuous for this
topology, unless it is either the identity map $u\mapsto u$ or complex
conjugation $u\mapsto\bar{u}$.

\aaa[Galois invariants and the minimal polynomial.]
We are interested in Galois invariants of trees and dessins, and we
would like a complete set of such invariants.  The minimal polynomial
is a basic example of a complete Galois invariant.

Let $a\in\Qbar$ be an algebraic number.  Of all non-zero
$f\in\mathbb{Q}[z]$ such that $f(a)=0$ there is only one that (i)~has
least degree, and (ii)~has top-degree coefficient~$1$.  This is called
the \emph{minimal polynomial $g_a(z)\in\mathbb{Q}[z]$} of $a$.

Suppose $b=\tilde{a}$ for some $u\mapsto\tilde{u}$ in $\Gamma$.  It is
easily proved that $g_a = g_b$.  Put another way, the minimal
polynomial $g_a$ is a \emph{Galois invariant} of $a\in\Qbar$.  Now
suppose $g_a=g_b$.  Does it follow that there is a $u\mapsto\tilde{u}$
in $\Gamma$, such that $b=\tilde{a}$.  If so, then we say that the
minimal polynomial is a \emph{complete} Galois invariant.  For use in
Proposition~\ref{prp-ubt-to-bt}, note that $f\mapsto f'$ for
$f\in\Qbar[z]$ is an example of something that is \emph{Galois
  covariant}.  Equivalently, the truth of the statement ``the
derivative of $f$ is $g$'' is Galois invariant (for $f, g\in\Qbar[z]$).

\begin{proposition}
\label{prp-mp-cgi}
Suppose $a\in\Qbar$.  Then the minimal polynomial
$g_a(z)\in\mathbb{Q}[z]$ is a complete Galois invariant of $a$.
\end{proposition}

The completeness of the minimal polynomial is a fundamental property
of the absolute Galois group.  It states that certain incomplete
automorphisms of $\Qbar$ can be indefinitely extended.

\aaa[Critical points and values.]
Suppose $f:\mathbb{C}\to\mathbb{C}$ is a polynomial function.  If
$f'(u)=0$ for $u\in\mathbb{C}$ we say that $u$ is a \emph{critical
  point} of $f$, and that $v=f(u)$ is a \emph{critical value}.  For
each $u\in\mathbb{C}$ let $v=f(u)$ and consider the polynomial
equation $f(z)-v=0$.  Writing
\[
f(z) = v + a_1(z-u) + a_2(z-u)^2 + \ldots + a_n(z-u)^n
\]
we see that $z=u$ is a simple root of $f(z)-v=0$ if and only if
$f'(u)\neq 0$.

Thus, provided $v\in\mathbb{C}$ is not a critical value of
$f:\mathbb{C}\to\mathbb{C}$, the fibre $f^{-1}(v)$ of $f$ consists of
$n$ distinct points, at each of which $f'$ is non-zero.  Using the
language of topology (\S\ref{s-bd-cs-pi}) we have that
$f:\mathbb{C}\to\mathbb{C}$ is a covering map away from the critical
values.

\aaa[Bipartite plane trees.]
The reader will need enough combinatorics to understand the following
result, which we will explain. Figure~\ref{fig-bpt-bpt}b shows a
bipartite plane tree. (By the way, Figure~\ref{fig-bpt-bpt}a is a
biased plane tree.)

\begin{proposition}
\label{prp-bpt-spp}
A bipartite plane tree is equivalent to an irreducible pair of
permutations such that $\alpha\beta$ has at most one orbit.
\end{proposition}

First, a word about equality.  We will say that two combinatorial
objects are equal if the one can be transformed into the other by
relabelling.  Thus, we are implicitly talking about equivalence
classes of labelled objects.  For example, any two graphs that have
only one vertex (and hence no edges) are equal, i.e.~belong to the
same equivalence class.

In this paper: (1)~A \emph{graph} $G$ is a set $V=V_G$ of
\emph{vertices} together with the \emph{edges} $E=E_G$, a set of
unordered pairs of vertices.  (2)~All graphs, trees and dessins will
have a finite number of vertices and edges.  (3)~A \emph{path} is a
sequence of edges of the form
$\{v_1,v_2\},\{v_2,v_3\},\ldots,\{v_{n},v_{n+1}\}$ such that the $v_i$
are distinct. (4)~A \emph{tree} is a graph where there is exactly one
path between any two distinct vertices. This condition allows the
no-vertex and one-vertex graphs as trees. (5)~A \emph{bipartite} graph
is one where (i)~$V_G$ is partitioned into two subsets, the
\emph{black} and \emph{white vertices}, and (ii)~each edge has a black
vertex and a white vertex. (6)~For consistency with $n=\deg f$, we let
$\deg X$ denote the number of edges in $X$, for $X$ a graph, tree or
(to be defined later) dessin.

In addition: (7)~The plane will always be $\mathbb{C}$, with its usual
counter-clockwise orientation.  (8)~A \emph{plane graph} will be a
graph that is drawn on the plane, with edges intersecting only at the
endpoints. (9)~Thus, a bipartite plane tree is (i)~a plane graph,
(ii)~with exactly one path between any two vertices, and (iii)~an
alternate black and white labelling of the vertices.

\begin{figure}[h]
\begin{centering}
Draft figures are at end of the document.\par
\end{centering}
\caption{(a)~A biased plane tree. (b)~The corresponding bipartite
  plane tree.}\label{fig-bpt-bpt}
\end{figure}

\begin{figure}[h]
\begin{centering}
Draft figures are at end of the document.\par
\end{centering}
\caption{(a)~The permutation $\alpha$. (b)~The permutation
  $\beta$. (c)~The permutation $\alpha\beta$.}\label{fig-perm-a-b-ab}
\end{figure}

\aaa[Pairs of permutations.]
First, a word about the figures.  Figure~\ref{fig-bpt-bpt}a shows a
biased plane tree $T$, and Figure~\ref{fig-bpt-bpt}b shows the
resulting bipartite plane tree $T'$ (which has at least one edge).
Figure~\ref{fig-perm-a-b-ab}a shows the permutation $\alpha$ on the
edges of $T'$ (and hence $T$), and Figure~\ref{fig-perm-a-b-ab}b the
permutation $\beta$.  Finally, Figure~\ref{fig-perm-a-b-ab}c show the
permutation $\alpha\beta$ on the edges of $T'$.  The key point of
Figure~\ref{fig-perm-a-b-ab}c is that $\alpha\beta$ is a
counterclockwise `two-step walk around' $T'$, which visits each side of each
edge exactly once.

Consider the edges in Figure~\ref{fig-bpt-bpt}b. Each edge $e$ has a
black vertex.  Rotating counterclockwise around that vertex we come to
another (or possibly the same) edge $e_1$.  We will write $e_1 =
\alpha(e)$.  Similarly, we define $\beta(e)$ by rotating
counterclockwise around the white vertex of
$e$. Figure~\ref{fig-perm-a-b-ab} parts (a) and (b) show $\alpha$ and
$\beta$ respectively.  Clearly, each bipartite plane tree $T$
determines a pair of permutations $(\alpha, \beta)$ on the edge set
$E=E_T$ of $T$.

Here's how the process can be reversed: (1)~A \emph{permutation} is a
bijection $\alpha:E\to E$ from a set to itself.  (2)~A \emph{pair of
  permutations} $P$ is an ordered pair $(\alpha_P, \beta_P)$ of
permutations of the same set $E=E_P$.  We call $E$ the \emph{edges} of
$P$. We require $E$ to be a finite set. (3)~We let $\Vb$ denote the
$\alpha$-orbits in $E$, and $\Vw$ the $\beta$-orbits. (4)~We let $V$
be the disjoint union of $\Vb$ and $\Vw$. We may need to relabel $\Vb$
or $\Vw$, for example when $E$ has only one element. (5)~Let $E'$ be
the pairs $\{\vb,\vw\}$ where $\vb$ and $\vw$ are orbits of the same
edge $e\in E$. (6)~We can, and will, identify $E$ and $E'$.  By
construction, there is at most one edge between two vertices.

This produces, from any pair of permutations $P$, (i)~a bipartite
graph $G_P$, together with (ii)~at each $v$ of $G_P$ a cyclic order on
the edges lying on that $v$.  Conversely, such data determines a pair
of permutations.  When is $G_P$ connected?  The reader is asked to
check:

\begin{notation}
$\langle\alpha,\beta\rangle$ is the group generated by $\alpha$ and
  $\beta$.
\end{notation}

\begin{definition}
A pair of permutations $P$ is \defemph{irreducible} if $E_P$ is either
empty or an orbit of $\langle\alpha_P,\beta_P\rangle$.
\end{definition}

\begin{proposition}
Let $P$ be a pair of permutations.  The graph $G_P$ is connected if
and only if $P$ is irreducible.
\end{proposition}

We now return to the proof of Proposition~\ref{prp-bpt-spp}.  Let $T$
be a bipartite plane tree, with $\deg T \geq 1$, and $P$ the
associated pair of permutations.  We have seen that $P$ is irreducible
and that $\alpha\beta$ has a single orbit on the edges of $P$.  Now
cut the plane along $T$ and, using rubber sheet geometry, deform the
cut plane until: (i)~it is a disc that is removed, and (ii)~the
boundary circle is divided into $2n$ arcs.

Because $\deg T\geq 1$, it has a vertex $v$ that lies on only one edge
$e$.  Suppose $v$ is black.  It follows that $\alpha(e) = e$.
Removing $e$ from $T$ glues back together two adjacent edges of the
boundary circle.  The result now follows if we can prove: (i)~the
hypothesis on $\alpha\beta$ implies that we can always find such an
edge, and (ii)~after removal of this edge the new $\alpha\beta$
still satisfies the hypothesis.  This will be done in
\cite{jfine-rpt-an}, or the reader can treat it as an exercise.

\section{Shabat polynomials and plane trees}

\aaa[Unbiased Shabat polynomials.]
We start with a summary of already known definitions and results. What
others have called a \emph{Shabat polynomial} we call, for clarity, an
\emph{unbiased Shabat polynomial}.  The same applies to \emph{dessins}
and \emph{unbiased dessins}.

\begin{definition}
An \defemph{unbiased Shabat polynomial} is a non-constant polynomial
function $f:\mathbb{C}\to\mathbb{C}$ together with an ordered pair
$(\vb, \vw)$ of distinct points in $\mathbb{C}$, such that if
$f'(u)=0$ then $f(u)\in\{\vb, \vw\}$.
\end{definition}

We call $\vb$ and $\vw$ the \emph{black} and \emph{white vertices}
respectively, and throughout will write $v_0 = (\vb + \vw)/2$ for the
midpoint of the line segment or \emph{edge} $[\vb, \vw]$ that joins
them.  Note that $\vb$ and $\vw$ need not be critical values.  For
example, $z\mapsto z$ is unbiased Shabat, for any distinct $\vb$ and
$\vw$.

\begin{definition}
A \defemph{change of coordinates} (on $\mathbb{C}$) is a map
$\psi:\mathbb{C}\to\mathbb{C}$ of the form $\psi(z) = az + b$, where
$a, b\in\mathbb{C}$ and $a \neq 0$.
\end{definition}

\begin{notation}
$\mathcal{S}'_n$ consists of all unbiased Shabat polynomials of degree
  $n$, modulo change of coordinates on both domain and range.  We
  write $\mathcal{S}' = \bigcup \mathcal{S}'_n$.
\end{notation}

Thus, each element $s$ of $\mathcal{S}'$ is an \emph{equivalence
  class} of unbiased Shabat polynomials.  This is why we need bias.
We use bias to (i)~reduce $s$ to a finite set of representatives, and
then (ii)~choose one of the representatives.  A polynomial
$f\in\Qbar[z]$ is much closer to algebraic numbers than an unbiased
Shabat equivalence class.  This is a great help (see
\S\ref{s-sh-choose}).

The reader is asked to check the following.  (1)~Change of coordinates
preserves the degree of~$f$.  (2)~Composition of functions induces a
group structure on the set of changes of coordinates.  (3)~If $f$ is
unbiased Shabat then so is $f\circ \psi$, with the same vertex pair
$(\vb, \vw)$.  (4)~Similarly, $\psi\circ f$ is also unbiased Shabat,
but with the pair $(\psi(\vb), \psi(\vw))$.  (5)~Given unbiased Shabat
$f$ there is a unique $\psi$ such that $(\vb, \vw)$ becomes $(-1, +1)$
when we apply $\psi$ to produce $\psi\circ f$.

\begin{definition}
$\mathcal{T}'_n$ consists of all non-empty bipartite plane trees with
  $n$ edges, and $\mathcal{T}' = \bigcup \mathcal{T}'_n$.
\end{definition}

As usual, $\mathcal{T}'$ is up to relabelling combinatorial
equivalence.  The next result is Grothendieck's bridge.  For a proof
see \cite{GGD}, \cite{LZ} or \cite{jfine-rpt-an}.

\begin{theorem}
\label{thm-ubt-bij}
The map $f\mapsto T_f=f^{-1}([\vb,\vw])$ induces a bijection between
$\mathcal{S}'_n$ and $\mathcal{T}'_n$.
\end{theorem}

$T_f$ is a combinatorial, and hence topological, description of $f$.
This is because $T_f$ can be used as the data for a gluing
construction, via covering spaces (see~\S\ref{s-bd-cs-pi}), that gives
a map $\mathbb{R}^2\to\mathbb{R}^2$ that is topologicaly equivalent to
$f:\mathbb{C}^2\to\mathbb{C}^2$.  For details see \cite{GGD} or
\cite{jfine-rpt-an}.

The theorem states that (i)~change of coordinates does not change the
combinatorial structure of $T_f$ (this is left to the reader),
(ii)~$T_f$ is a bipartite plane tree, and (iii)~we can reconstruct $f$
from $T_f$, up to change of coordinates.  Put another way, topology
determines geometry.  In \S\ref{s-sh-bsp} we add bias to both $f$ and
$T_f$.  We do this so that $f$ to be reconstructed exactly, without
the change of coordinates indeterminacy.

The following are key for the usefulness of the bridge.  For a proof
see \cite{GGD}, \cite{LZ} or \cite{jfine-rpt-an}.

\begin{lemma}
Each equivalence class $s=[f]$ in $\mathcal{S}'$ has at least one
element $f_1$ that lies in $\Qbar[z]$.
\end{lemma}

\begin{theorem}
\label{thm-ubs-ffl}
$\Gamma$ acts on $\mathcal{S}'_n$, and its action on $\mathcal{S}'$ is
faithful.
\end{theorem}

\aaa[Goals.]
The bijection between $\mathcal{S}'$ and $\mathcal{T}'$ produces an
action of the absolute Galois group $\Gamma$ on $\mathcal{T}'_n$.
Understanding this action combinatorially, without going over the
bridge into algebraic numbers, would help us understand $\Gamma$.
Some first steps are to find Galois invariants of $\mathcal{T}'$, and
to understand the decomposition of $\mathcal{T}'_n$ into orbits.

The main goal is understanding $\Gamma$.  For us biased and unbiased
objects are a means to an end. The main idea of this paper is that the
goal is better reached by using biased objects.





\aaa[\label{s-sh-choose}Choosing $f$ in $s\in\mathcal{S}'_n$.]
We want Galois invariants of $s\in\mathcal{S}'_n$.  If each element of
$\mathcal{S}'_n$ were a polynomial $f\in\Qbar[z]$ then the minimal
polynomials $g_i(z)\in\mathbb{Q}[z]$ of the coefficients $a_i$ of $f$ would
be Galois invariants of $f$ and hence of $T_f$.  But each element $s$
of $\mathcal{S}'_n$ is an equivalence class of unbiased Shabat
polynomials, not a single such polynomial.

If we could in an Galois invariant way choose an $f$ in $s$, then we
could use that $f$ instead of $s$.  This seems not to be possible, but
we can come close enough.  We can define a non-empty finite subset of
$s$, in a Galois invariant manner (see also~\S\ref{s-cnc-cr}).
Choosing an element from this subset we call the process of
\emph{biasing $f$} (in its equivalence class).

Suppose $f$ is unbiased Shabat.  Let $f_1$ be $\psi\circ f \circ
\eta$, for changes of coordinates $\psi$ and $\eta$.  We want to
choose $\psi$ and $\eta$ so $f_1$ is fixed, up to a finite choice.
Already, the reader has checked that there is a unique $\psi$ such
that $(-1,1)$ is the black-white vertex pair associated with $f_1$.
The uniqueness is important.  We now need a condition that determines
$\eta$.

Let $f_0$ be $\psi\circ f$.  It has vertex pair $(-1, 1)$.  Now
consider the equation $f_0(z) = 0$.  Counted with multiplicity, this
has $n$ roots.  Let $u\in\mathbb{C}$ be one of them.  If $f_0'(u)=0$
then, by the Shabat condition, $f_0(u) \in\{-1,+1\}$.  Thus,
$f_0'(u)\neq 0$ and $f_0(z)=0$ has exactly $n$ distinct roots.

Recall that $f_1=f_0\circ \eta$. Assume that $f_0(u)=0$.  This is the
finite choice.  The change of coordinates $\eta$ has two degrees of
freedom.  If $\eta(z) = az + u$ then $f_1(0) = f_0(u) = 0$.  Assume
$\eta$ has this form.  This leaves $a$ to be determined.  Now consider
$f_1'(0)$.  By the chain rule we have $f_1'(0) = a f_0'(u)$.  We have
just seen that $f_0'(u)\neq 0$ and so we can write $a=1/f_0'(0)$ to
give $f_1'(0) = 1$.  The bias is a choice of one of the $n$ roots of
$f_0(z)=0$, or equivalently $f(z)=v_0$, where as usual $v_0=(\vb +
\vw)/2$.

\aaa[Applying the choosing process.]
Here we summarize the \S\ref{s-sh-choose}, and prepare for bias.  The
previous discussion shows:

\begin{proposition}
\label{prp-uni-cc}
Suppose $f$ is unbiased Shabat, with vertex pair $(\vb,\vw)$.  Suppose
also that $u\in\mathbb{C}$ is a root of $f(z) = v_0$.  Then there is a
unique pair $\psi, \eta$ of changes of coordinates such that
(i)~$\psi(\vb) = -1$ and $\psi(\vw) = +1$, (ii)~$\eta(0) = u$, and
(iii)~$(\psi\circ f\circ\eta)'(0) = 1$.
\end{proposition}

Note that $\psi$ is affine linear, so $\psi(v_0) = ((-1) + (+1))/2 =
0$ and thus $(\psi\circ f\circ \eta)(0) = 0$.

\begin{proposition}
\label{prp-ubt-to-bt}
Let $f$ and $u$ be as above, and let $f_1$ be the resulting $\psi\circ
f \circ \eta$.  Then:
\begin{enumerate}
\item $f_1$ is biased Shabat, as in Definition~\ref{dfn-bsp} below.
\item If $f\in\Qbar[z]$ then $f_1$ is also in $\Qbar[z]$.
\item Applied to $\tilde{f}$ and $\tilde{u}$ the construction yields
  $\tilde{r}$, where $r=f_1$.  In other words, the construction is
  Galois covariant.
\end{enumerate}
\end{proposition}

\begin{proof}
Parts (1) and (3) are left to the reader.  \emph{Biased Shabat} is
defined as it is, to make (1) true.  Part~(3) is needed for the proof
of Theorem~\ref{thm-bt-bij}.  Its proof is purely formal.

The proof of~(2) has a tricky special case. Suppose $f\in\Qbar[z]$.
By Lemma~\ref{lem-cv-in-qb} below the critical values of $f$ lie in
$\Qbar$.  If $f$ has two critical values then $\vb, \vw \in \Qbar$.
This is enough to ensure $\psi\in\Qbar[z]$, as $\Qbar$ is a field.
Similarly, $u\in\Qbar$ as $f(u)=v_0$ and $\Qbar$ is algebraically
closed, and thus $\eta\in\Qbar[z]$.  As $f, \psi, \eta \in \Qbar[z]$
it follows that $f_1=\psi\circ f\circ \eta \in \Qbar[z]$.

We now have to deal with the special cases.  The first is easy.  If
$f$ has no critical values then it is a change of coordinates.  We ask
the reader to check that the process results in $f_1(z)=z$.

Now assume $f$ has exactly one critical value, say $\vb$.  This
requires a trick. Consider $T_f$.  By Theorem~\ref{thm-ubt-bij}, it is
a plane tree.  By assumption, the white vertices are not critical
points, and so lie on only one edge.  Thus, $T_f$ is an $n$-pointed
star, with a black vertex at the centre.  But $p(z) = z^n$ with $(0,
1)$ also gives $T_f$ and so, again by Theorem~\ref{thm-ubt-bij}, some
change of coordinates will take $f$ to $p$.  We are now out of the
special case, and the previous argument produces a $p_1\in\Qbar[z]$.
By uniqueness of the change of coordinates (see
Proposition~\ref{prp-uni-cc}), we have $f_1 = p_1$.  The author does
not see how to avoid using Theorem~\ref{thm-ubt-bij}, or something
similar.
\end{proof}

\begin{lemma}
\label{lem-cv-in-qb}
Suppose $f\in\Qbar[z]$. Then the critical values of $f$ lie in
$\Qbar$.
\end{lemma}

\begin{proof}
Suppose $\deg f'\geq 1$, and $f'(u) = 0$. It follows that $u\in\Qbar$
(as $\Qbar$ is algebraically closed) and then $v = f(u)\in\Qbar$ (as
$\Qbar$ is a field).  The remaining case, $f(z)$ constant, is trivial.
\end{proof}

\aaa[\label{s-sh-bsp}Biased Shabat polynomials.]
Here we add bias to the definitions, and thereby remove equivalence
classes from the polynomial end of the bridge.  This will give new
Galois invariants.

\begin{definition}
\label{dfn-bsp}
A \defemph{biased Shabat polynomial} is a polynomial function
$f:\mathbb{C}\to\mathbb{C}$ such that (i)~if $f'(u)=0$ then
$f(u)\in\{-1, +1\}$, (ii)~$f(0)=0$, and (iii)~$f'(0)=1$.
\end{definition}

\begin{notation}
$\mathcal{S}_n$ is all biased Shabat polynomials of degree $n$, and
  $\mathcal{S} =\bigcup\mathcal{S}_n$.
\end{notation}

\begin{proposition}
If $f$ is biased Shabat then $f\in\Qbar[z]$.
\end{proposition}

\begin{proof}
This follows from the unbiased result.  Think of $f$ as unbiased
Shabat. By Theorem~\ref{thm-ubt-bij} there is a change of coordinates
$(\psi, \eta)$ that produces from $f$ an unbiased $\psi\circ f \circ
\eta = f_1\in\Qbar[z]$.  Now bias $f_1$, choosing $\eta^{-1}(0)$ as
the solution $u$ of $f_1(z) = v_0$.  By
Proposition~\ref{prp-ubt-to-bt} the result $f_2$ lies in $\Qbar[z]$.
By Proposition~\ref{prp-uni-cc} the change of coordinates that does
this is unique.  So it must be $(\psi^{-1},\eta^{-1})$ and thus $f =
f_2$ lies in $\Qbar[z]$.
\end{proof}

\begin{corollary}
$\Gamma$ acts on $\mathcal{S}_n$, by acting on the coefficients.
\end{corollary}

\begin{proof}
This is because the biased Shabat conditions are Galois invariant.
For example, if $f'(u)=0$ then $\tilde{f}'(\tilde{u}) = \tilde{0} =
0$, and vice versa.  Similarly, $f(u) = -1$ if and only if
$\tilde{f}(\tilde{u}) = -1$.  The same applies to $f(u) = +1$, $f(0) =
0$ and $f'(0)= 1$.
\end{proof}

\begin{corollary}
The action of $\Gamma$ on $\mathcal{S}$ is faithful.
\end{corollary}

\begin{proof}
The forget-bias map $\mathcal{S}\to\mathcal{S}'$ is surjective, and
consistent with the Galois action.  The Galois action is faithful on
$\mathcal{S}'$, by Theorem~\ref{thm-ubs-ffl}.
\end{proof}

\aaa[\label{s-sp-gi}Galois invariants.]
Recall (see Proposition~\ref{prp-mp-cgi}) that each $a\in\Qbar$ has a
minimal polynomial $g_a(z)\in\mathbb{Q}[z]$, and that $g_a$ is a
complete Galois invariant for $a$.  Let $f(z) = a_0 + a_1z + \ldots +
a_nz^n$ be a polynomial in $\Qbar[z]$.  Clearly, the sequence
$g_i(z)\in\mathbb{Q}[z]$ of the minimal polynomials of the
coefficients $a_i$ is a Galois invariant of $f$.  Thus we obtain many
Galois invariants of biased Shabat polynomials.  Of course, for $f$
biased Shabat $a_0=0$ and $a_1=1$, and so $g_0$ and $g_1$ are constant
on $\mathcal{S}$.

On $\Qbar[z]$, the sequence of minimal polynomials is not a complete
Galois invariant. For example, all coefficients of $f_-(z) =
\sqrt{2}(1 - z)$ and $f_+(z) = \sqrt{2}(1 + z)$ have $g(z)=z^2-2$ as
their minimal polynomial. But $f_-(1) = 0 \in \mathbb{Q}$ while
$f_+(1) = 2\sqrt{2}\notin\mathbb{Q}$.  The author suspects that there
are distinct $f_1,f_2\in\mathcal{S}_n$ with $g_{1,r}(z) = g_{2,r}(z)$
for all $r\leq n$.

\aaa[Biased plane trees.]
Recall that unbiased Shabat polynomials correspond to bipartite plane
trees.  For biased polynomials, we want a similar corresponding
definition.  Let $f$ be biased Shabat.  Consider
$T_f=f^{-1}([-1,1])$. By forgetting the bias we see, as before, that
$T_f$ is a plane tree with a bipartite colouring of the
vertices. Because $f(0) = 0\in [-1, 1]$, we have $0\in T_f$.  In fact,
each of the $n$ edges has an interior point $c$ such that $f(c) = 0$,
and so $0$ lies on a single edge $e_f$ of $T_f$.

Thus, even in the unbiased case, the choice of a root of $f(z) = v_0$
is equivalent to the choice of an edge in $T_f$.  If $f$ is biased
then $f(0)=0$ is the chosen root.  This gives rise to:

\begin{definition}
A \defemph{biased plane tree} $T$ is a bipartite plane tree with a
chosen edge $e_T$.
\end{definition}

Now draw the tree, and an arrow, black vertex to white, on the chosen
edge. This, by itself, is enough to determine the colour of all other
vertices of the tree (see Figure~\ref{fig-bpt-bpt}), and we still have
a chosen edge.  Thus, the previous definition is equivalent to:

\begin{definition}
A \defemph{biased plane tree} is a plane tree with an arrow (the bias)
along one edge.
\end{definition}

We can now state the biased analogue of Theorem~\ref{thm-ubt-bij}.

\begin{notation}
$\mathcal{T}_n$ is all biased plane trees with $n$ edges, and
  $\mathcal{T}=\bigcup\mathcal{T}_n$.
\end{notation}

\begin{theorem}
\label{thm-bt-bij}
The map $f\mapsto T_f=f^{-1}([-1, -1])$ induces a bijection between
$\mathcal{S}_n$ and $\mathcal{T}_n$.
\end{theorem}

\begin{proof}
Think of a biased $f$ as an unbiased $f$, together with a root $c$ of
the equation $f(z)=v_0$.  Now use the bijection between
$\mathcal{S}'_n$ and $\mathcal{T}'_n$ provided in
Theorem~\ref{thm-ubt-bij}.  We can use $c$ to select an edge on $T_f$,
and vice versa. This lifts the bijection to $\mathcal{S}_n$ and
$\mathcal{T}_n$.
\end{proof}

\aaa[Rooted plane trees and Catalan numbers.]
A biased plane tree is the same as a rooted plane tree, as used in
linguisitics and computer science for parse and syntax trees, except
that a rooted plane tree need not have any edges.  Thus, \emph{biased
  plane tree} is a shorthand for \emph{rooted plane tree with at least
  one edge}.  For us, the black-white alternation of vertices along
edges is important, as is the presently mysterious Galois action.

It is well known that the number of rooted plane trees with $n$ edges
is the $n$-th Catalan number.  As $\Gamma$ acts faithfully on
$\mathcal{T}$, it also acts faithfully on any set that is in bijection
with rooted plane trees.  There are many interesting examples of
such~\cite{rstan-cn}.  This will be explored further
in~\cite{jfine-rpt-an}.

\section{Dessins}

\aaa[Overview.]
In the previous section we introduced bias to solve a geometric
problem, namely that unbiased $T$ determines $f$ only up to change of
coordinates.  In this section we add bias to solve a combinatorial
problem, namely that the Cartesian product of two trees is not a tree.
To do this we also have to generalise tree to dessin.  We use the same
concept of bias.  This process puts a Galois invariant lattice
structure on the set of biased dessins.  We can use this
(see~\S\ref{s-cnc-tgi}) to define new Galois invariants from old.

\aaa[Unbiased dessins.]
Recall (Proposition~\ref{prp-bpt-spp}) that a bipartite plane tree is
equivalent to an irreducible pair $P=(\alpha, \beta)$ of permutations,
such that $\alpha\beta$ has at most one orbit.  Sets have a Cartesian
product, and something similar can be done for pairs of permutations.

\begin{definition}
For pairs of permutations $P_1$ and $P_2$ the \defemph{product}
$P_1\times P_2$ has edge set $E_1\times E_2$ and permutations
$\alpha((e_1, e_2)) = (\alpha_1(e_1), \alpha_2(e_2))$, and similarly
for $\beta$.
\end{definition}

The product $T = R\times S$ of two pairs of permutations is also a
pair of permutations.  Even when $R$ and $S$ are irreducible, $T$ may
be reducible.  For example, $R\times R$ is reducible if $R$ has two or
more edges.  This is because its diagonal $\{(e,e)| e \in R\}$ is
irreducible, but is not the whole of $R\times R$.  However, $R\times
S$ always decomposes into irreducibles, each of which is an
$\langle\alpha,\beta\rangle$ orbit.

We generalise the concept of unbiased plane tree as follows:

\begin{definition}
An \defemph{unbiased dessin} is an irreducible pair $D$ of
permutations, where $D$ has at least one edge.
\end{definition}

Note that each product of unbiased dessins, which may be reducible,
has a unique decomposition into unbiased dessins.


\begin{notation}
$\mathcal{D}'_n$ is all unbiased dessins with $n$ edges, and
  $\mathcal{D}'=\bigcup\mathcal{D}'_n$.
\end{notation}

\aaa[Biased dessins.]
We have just seen that the product $T=R\times S$ of two unbiased
dessins is sometimes reducible, and so not a dessin.  We will choose a
component of $T$ as follows:

\begin{definition}
A \defemph{biased dessin} $D$ is an irreducible pair of permutations,
together with a chosen edge $e_D$ of $D$.
\end{definition}

\begin{notation}
$\mathcal{D}_n$ is all biased dessins with $n$ edges, and
  $\mathcal{D}=\bigcup\mathcal{D}_n$.
\end{notation}

\begin{definition}
\label{dfn-lat-join}
The \defemph{join} $T=R\vee S$ of two biased dessins is the
$\langle\alpha_T,\beta_T\rangle$ orbit of $(e_R, e_S)$ in the product
$R\times S$, with chosen edge $e_T = (e_R, e_S)$.
\end{definition}

\aaa[Morphisms.]
Suppose $R$ and $S$ are pairs of permutations.  A \emph{morphism}
$\psi:R\to S$ is a set map $\psi:E_R\to E_S$ such that
$\psi\circ\alpha_R = \alpha_S\circ\psi$ and similarly for $\beta$.  We
use the same concept for unbiased dessins.

\begin{definition}
A \defemph{morphism} $\psi:R\to S$ of biased dessins is a pair of
permutations morphism, call it $\psi$, such that $\psi(e_R) = e_S$.
\end{definition}

Each biased dessin is $\langle\alpha,\beta\rangle$ irreducible, and
morphisms respect the chosen edge.  From this it easily follows that:

\begin{lemma}
For any two biased dessins $R$ and $S$ there is at most one morphism
$\psi:R\to S$.
\end{lemma}

\begin{notation}
For biased dessins we write $R\to S$ if there is a morphism $\psi:R\to
S$.
\end{notation}

Thus we can think of $R\to S$ either as a boolean relation between $R$
and $S$, or as the combinatorial structure that makes this relation
true.  Clearly, $R\to S$ is a partial order.  In \cite{jfine-dlbd} we
will prove:

\begin{theorem}
The relation $R\to S$ gives $\mathcal{D}$ a lattice structure, with
join as in Definition~\ref{dfn-lat-join}.
\end{theorem}

\aaa[Marked Belyi pairs.]
Extending the bijection between $\mathcal{S}$ and $\mathcal{T}$, there
is a concept of marked Belyi pair such that:

\begin{notation}
$\mathcal{B}_n$ is all marked Belyi pairs of degree $n$, and
  $\mathcal{B}=\bigcup\mathcal{B}_n$.
\end{notation}

\begin{theorem}
$\Gamma$ acts on $\mathcal{B}_n$.  The action on $\mathcal{B}$ is
  faithful.
\end{theorem}

\begin{theorem}
The map $f\mapsto D_f=f^{-1}([-1, -1])$ induces a bijection between
$\mathcal{B}_n$ and $\mathcal{D}_n$.
\end{theorem}

\begin{theorem}
\label{thm-ls-gi}
The lattice structure on $\mathcal{B}$ is Galois invariant under this
bijection.
\end{theorem}

The proof of these results, and the definition of marked Belyi pair,
will be given in \cite{jfine-dlbd}. The proof can be done, as in
Theorem~\ref{thm-bt-bij}, by adding bias to the corresponding unbiased
result.

\aaa[\label{s-cnc-tgi}The tower of Galois invariants.]
We can use the lattice structure on $\mathcal{B}$ to produce new
Galois invariants from old.  Let $h:\mathcal{B}\to\mathcal{V}$ be any
Galois invariant, such as the degree (number of edges), or the
partition triple (see Proposition~\ref{prp-pt-gi}).  If
$R\in\mathcal{B}$ is Galois invariant then so is the function
$X\mapsto h(R\vee X)$.  Now suppose $S\subset\mathcal{B}$ is a Galois
invariant subset.  Using formal sums (see below) we have that
\begin{equation}
\label{eqn-gi-fs}
h_S(X) = \sum\nolimits_{Y\in S}\> [h(Y\vee X)]
\end{equation}
is also Galois invariant.  Something similar can be done with
$S\subset \mathcal{B}\times \mathcal{B}$ and so on.

\begin{definition}
A \defemph{formal sum} (on a set $\mathcal{V}$ of \defemph{values}) is
a map $m:\mathcal{V}\to\mathbb{Z}$ that is zero outside a finite
subset of $\mathcal{V}$.
\end{definition}

\begin{notation}
We write $m:\mathcal{V}\to\mathbb{Z}$ as $\sum m(v)[v]$, perhaps
omitting terms where $m(v)=0$.
\end{notation}

Conversely, if $h:\mathcal{B}\to \mathcal{V}$ is a Galois invariant
and $R\in\mathcal{B}$ then
\begin{equation}
\label{eqn-gi-ss}
S_R = \{Y | h(Y) = h(R) \} \subset \mathcal{B}
\end{equation}
is also Galois invariant, and so can be used as in the previous
paragraph.

In this way, by alternating Galois invariant maps
$\mathcal{B}\to\mathcal{V}$ as in (\ref{eqn-gi-fs}), and finite
subsets $S\subset\mathcal{B}$ as in (\ref{eqn-gi-ss}), we can
construct a tower of Galois invariants.  For completeness, this
process should be extended to include $\mathcal{B}\times\mathcal{B}$
and so on.  The process produces formal sums of formal sums and so on.
One wants as many invariants as possible, while at the same time
managing the duplication and redundancy that results.  These matters
will be further discussed in~\cite{jfine-gibd}.

\aaa[\label{s-bd-cs-pi}Covering spaces and $\pi_1(\hat{X})$.]
From now until the end of this section we will rely on some concepts
and results from topology, which we will use to motivate the
definition of the decorated lattice $\mathcal{L}$ and to outline the
proof of its Galois invariance.  This results in many new invariants,
to which the just described tower construction can be applied.  What
follows is intended for experts in dessins.  Others may find it hard.

A map $f:Y\to X$ of topological spaces is a \emph{covering map} if
$f^{-1}(U)$ is the disjoint union of copies of $U$, for small enough
open subsets $U$ of $X$.  The Shabat condition ensures that
$f:\mathbb{C}\to\mathbb{C}$ is a covering map away from $\vb$ and
$\vw$.

The \emph{fundamental group} $\pi_1(X, x_0)$ consists of all
continuous maps $p:[0,1]\to X$ with $p(0)=p(1)=x_0$, considered up to
homotopy equivalence.  Following first path $p$ and then path $q$
gives the group law on $\pi_1(X, x_0)$.  This definition relies on the
choice of a base point $x_0$ (and each path from $x_0$ to $x_1$
induces an isomorphism between $\pi_1(X, x_0)$ and $\pi_1(X, x_1)$).
The subgroups of $\pi_1(X, x_0)$ are related to the covers of $X$.

A \emph{pointed topological space} $\hat{X}$ is a topological space
$X$ together with a base point $x_0$.  We let $\pi_1(\hat{X})$ denote
$\pi_1(X, x_0)$.  Suppose $f:Y\to X$ is a covering map, with
$f(y_0)=x_0$.  Write $\hat{Y}$ for the pointed topological space $(Y,
y_0)$ and similarly for $\hat{X}$.  We will say that
$f:\hat{Y}\to\hat{X}$ is a \emph{pointed covering map}.

\begin{theorem}
\label{thm-cpc-sg-pi}
Provided $\hat{X}$ is connected and locally path connected, the
connected pointed covers $f:\hat{Y}\to\hat{X}$ correspond to the
subgroups $\pi_1(\hat{X})$, and vice versa.
\end{theorem}

This theorem applies in our situation, with $X =
\mathbb{C}\setminus\{-1, +1\}$ and $x_0=0$.  Each biased dessin $R$
produces a finite pointed cover $\hat{Y}_R\to \hat{X}$.  The relation
$R\to S$ on biased dessins, translated to topology, is equivalent to:
The pointed covers $\hat{Y}_R\to\hat{X}$ and $\hat{Y}_S\to\hat{X}$ are
such that (i)~there is a pointed cover map $\hat{Y}_R\to\hat{Y}_S$,
and (ii)~the composite $\hat{Y}_R\to\hat{Y}_S\to\hat{X}$ is
$\hat{Y}_R\to\hat{X}$.

From this, and standard results that produce a Belyi pair from a
finite cover of $X$, it follows that the relation $R\to S$ on biased
dessin (and hence the lattice structure) is Galois invariant
(Theorem~\ref{thm-ls-gi}).  Biased dessins (and maps between them)
correspond to finite pointed covers of $\mathbb{C}\setminus\{-1, +1\}$
(and maps between them).

\aaa[$\pi_1(\hat{X})$ and the lattice structure.]
By design, each Shabat polynomial gives a covering space (away from
$\vb$ and $\vw$), with a finite number of sheets.  The same goes for
Belyi pairs and $\PC{1}$ less three points.  Therefore, once bias has
provided base points, we can apply Theorem~\ref{thm-cpc-sg-pi}.

Suppose $H_R$ and $H_S$ are subgroups of $G=\pi_1(\hat{X})$.  In this
situation both $H_R\cap H_S$ and $\langle H_R,H_S\rangle$ (the
subgroup generated by $H_R$ and $H_S$) are subgroups of $G$.  This
puts an order lattice structure on the subgroups of $G$.  The
construction of the join $R\vee S$ of two biased dessins
(see~Definition~\ref{dfn-lat-join}) corresponds to $H_R\cap H_S$ in
$\pi_1(X, x_0)$, where $X = \mathbb{C}\setminus\{-1,+1\}$ and $x_0 = 0
\in X$.

\aaa[The partition triple.]
We have just, via covering spaces, outlined why the lattice structure
on $\mathcal{B}$ is Galois invariant.  This uses the global structure
of biased dessins $R$ and $S$ to define the relation $R\to S$.  If we
have $R\to S$ then there is also significant local structure that is
Galois invariant.  We will now outline how this produces from
$\mathcal{B}$ the decorated lattice $\mathcal{L}$.

Recall that $R$ has permutations $\alpha_R$ and $\beta_R$ acting on
the edges $E_R$ of $R$.  Recall also that each black vertex of $R$ is
an $\alpha_R$ orbit in $E_R$.  Thus, $\alpha$ partitions $E_R$ into
orbits, and hence produces a partition $p_{R,\alpha}$ of $n = \deg R$.
We can similarly define $p_{R,\beta}$ and $p_{R,\gamma}$, where
$\gamma = (\alpha\beta)^{-1}$ gives what is called the monodromy
around $\infty\in\PC{1}$.  The following is easy and already known.

\begin{proposition}
\label{prp-pt-gi}
The \defemph{partition triple}
$(p_{R,\alpha},p_{R,\beta},p_{R,\gamma})$ is a Galois invariant of
$R\in\mathcal{B}$.
\end{proposition}

The decoration that gives $\mathcal{L}$ is a relative form of the
partition triple.  First a review.  Let $D_1$ be the unique
single-edged biased dessin.  Given $R\to D_1$ we have marked Belyi
pair $M_R\to\PC{1}$.  Further, the partition $p_{R,\alpha}$ gives
Galois invariant information about the monodromy of $M_R\to\PC{1}$
around $-1\in\PC{1}$, and similarly for $p_{R,\beta}$ and
$p_{R,\gamma}$ around $+1$ and $\infty$ respectively.

Now suppose we have $R\to S\to D_1$.  Each say black vertex $v_r$ of
$R$ maps to a black vertex $v_s$ of $S$ (then to the the black vertex
of $v_b$ of $D_1$, which is what gives $v_r$ and $v_s$ their colour).
Each vertex $v_r$ of $R$ has a multiplicity $\mult v_r$ (number of
edges that meet $v_r$).  The numbers $\mult v_r$, for all $v_r$
mapping to $\vb$, give the partion $p_{R,\alpha}$.

The vertex $v_r$ also maps to a vertex $v_s$ on $S$.  This gives
additional information to record.

\aaa[\label{s-bd-dl}Decorating the lattice.]
Let $\mathcal{L}'$ be $\mathcal{B}$ considered as an abstract lattice,
whose elements we will call \emph{nodes}.  Each node $R$ is secretly a
biased dessin, but for Galois purposes we are not allowed to look
inside $R$ and see the biased dessin.  The underlying biased dessin is
without Galois significance, which is why we keep it secret.  However,
some information does emerge.

The \emph{decoration} of $\mathcal{L}'$ consists of: (1)~For each node
$R$ of $\mathcal{L}'$ a finite set $V_R$, called the \emph{vertices}
of $R$. (2)~A map $\mult:V_R\to\mathbb{N}^+ =\{n>0\}$.  (3)~Whenever
$R\to S$, which now means the abstract partial order on
$\mathcal{L}'$, there is a map $V_R\to V_S$.

\begin{definition}
The \defemph{decorated lattice of biased dessins} $\mathcal{L}$ is
$\mathcal{L}'$ decorated as above.
\end{definition}

We consider two decorations of a lattice to be equal if they are the
same after relabelling, or in other words are related by bijections on
the vertex sets $V_R$.  Our decoration of $\mathcal{L}'$ has special
properties, such as (i)~the maps $V_R\to V_S$ commute, and (ii)~if
$v_r\mapsto v_s$ under $V_R\to V_S$ then $\mult v_r$ divides $\mult
v_s$.  We don't need these properties in this paper.  But we do care
about automorphisms.

\begin{definition}
An \defemph{automorphism $\psi$ of $\mathcal{L}$} consists of a
lattice isomorphism $\psi:\mathcal{L}'\to\mathcal{L}'$, together with
maps $\psi:V_R\to V_{\psi(R)}$, such that (i)~the composition $V_R\to
V_{\psi(R)}\rightarrowmult\mathbb{N}^+$ is equal to
$V_R\rightarrowmult\mathbb{N}^+$, and (ii)~if $R\to S$ then the
compositions $V_R\to V_{\psi(R)}\to V_{\psi(S)}$ and $V_R\to
V_S\to V_{\psi(S)}$ are equal.
\end{definition}

Recall that $\mathcal{L}'$ is an abstract lattice, each of whose nodes
has secretly associated with it a biased dessin.  Suppose $\psi$ is
automorphism of $\mathcal{L}$ and $R$ is a node of $\mathcal{L}$.  Let
$U$ and $\psi(U)$ be the biased dessins secretly associated with $R$
and $\psi(R)$.  It is not required that $\psi$ induce a bijection
between the edges of $U$ and those of $\psi(U)$.  Recall that only two
elements of $\Gamma$ act continuously on $\Qbar\subset\mathbb{C}$
(see~\S\ref{s-i-an}).  This might make it impossible to construct a
bijection on the edges.

What $\psi$ must do is preserve certain geometric relations between
elements of $\mathcal{B}$.  The lattice isomorphism
$\psi:\mathcal{L}'\to\mathcal{L}'$ comes from global properties.  The
$V_R$, $\mult:V_R\to\mathbb{N}^+$ and $V_R\to V_S$ come from local
geometric properties.

\begin{notation}
$\Gamma' = \Aut(\mathcal{L})$, the automorphism group of
  $\mathcal{L}$.
\end{notation}

The bottom element $D_1$ of $\mathcal{L}$ has three vertices, which we
denote by $\vb$, $\vw$ and $v_\infty$.  Each has multiplicity
one. Given a node $R$ of $\mathcal{L}$, the map $V_R\to V_{D_1} =
\{\vb,\vw,v_\infty\}$ partitions $V_R$ into black, white and
at-infinity vertices.  The map $V_R\to\mathbb{N}^+$, restricted to
each of these subsets, then gives the partition triple.

Each permutation of $\vb,\vw,v_\infty$ induces an automorphism of
$\mathcal{L}$.  The following, given Theorem~\ref{thm-bt-bij}, is not
hard.  Its proof will be given in \cite{jfine-dlbd}.

\begin{notation}
$\Gamma'_0$ is the subgroup of $\Gamma'$ that fixes $V_{D_1}$.
\end{notation}

\begin{theorem}
The absolute Galois group $\Gamma$ is a subgroup of $\Gamma'_0$.
\end{theorem}

At present, there is not evidence or a proof strategy for:

\begin{conjecture}
$\Gamma = \Gamma'_0$.
\end{conjecture}

\section{Conclusion}

\aaa[Summary.]
We have seen that adding bias to dessins brings many benefits.
(1)~Galois invariants can be defined directly from biased Shabat
polynomials, say via minimal polynomials.  (2)~Biased dessins have a
Galois invariant lattice structure, which can be use to help build a
tower of Galois invariants.  (3)~Biased plane trees are counted by the
Catalan numbers, which brings connections to many other parts of
mathematics. (4)~The decorated lattice $\mathcal{L}$ of biased dessins
is the ground for the definition of new Galois invariants, which
generalise the partition triple. (5)~The simply defined subgroup
$\Gamma'_0$ of $\Aut(\mathcal{L})$ contains, and might equal, the
absolute Galois group $\Gamma$.

To this list we add: (6)~Each $\psi\in\Gamma'_0$ induces a bijection
$\psi:\mathcal{A}\to\mathcal{A}$, where $\mathcal{A}\subset\Qbar$ are
the coefficients that appear in $\mathcal{S}$.  (7)~We have additional
structures and conjectures that can be explored using computer
calculations.  The purely combinatorial calculations might be easier.

Benefit~(6) is importantant because $\psi\in\Gamma'_0$ will induce,
and hence come from, a $\psi\in\Gamma$ just in case
$\psi:\mathcal{A}\to\mathcal{A}$ respects all algebraic relations that
exist between the elements of $\mathcal{A}$.  This makes $\mathcal{A}$
a potentially interesting object of study.

\aaa[Two cultures.]
The minimal polynomial and the partition triple are both Galois
invariants, but very different in character.  The one is algebraic,
the other combinatorial.  They also apply to different types of
object, namely elements of $\Qbar$ and $\mathcal{B}$ respectively.
Thus, each type of object has its own type of Galois invariant.

The introduction of bias destroys this dichotomy.  Each biased Shabat
polynomial is, via the bridge, a biased plane tree and vice versa.  As
a biased Shabat polynomial it has `minimal polynomial' style
invariants.  As a plane tree it has `partition triple' style
invariants.

Suppose we have a complete set $\mathcal{X}$ of Galois invariants on,
say, the algebraic number side.  This means that \emph{any} Galois
invariant on the dessins side can be expressed using the $\mathcal{X}$
invariants.  The bridge will become more useful if we can produce sets
of invariants $\mathcal{X}$ and $\mathcal{Y}$, one at each end the
bridge, that are \emph{aligned}.  By this I mean, for example, that
$\mathcal{X}(f)$ and $\mathcal{Y}(T_f)$ are linear functions of each
other.  The author hopes to discuss this further in~\cite{jfine-gibd}.

\aaa[\label{s-cnc-ubgi}Unbiased Galois invariants.]
We have seen that biased Shabat polynomials and plane trees have many
Galois invariants, coming from the coefficients of $f$ and the lattice
structure on $\mathcal{B}$ respectively.  Suppose, however, that our
situation requires the study of unbiased objects.  What now?

Formal sums allow Galois invariants to descend, solving this problem.

\begin{proposition}
If $h$ is a biased Galois invariant then
\[
h_\Sigma(X) = \sum\nolimits_{Y'=X} [h(Y)]
\]
is an unbiased Galois invariant.  Here $Y'$ means $Y$ without its
bias.
\end{proposition}

\begin{proof}
The set $S_X = \{Y\in\mathcal{B} | Y'=X\}$ is finite, and Galois
covariant.
\end{proof}

This process can be thought of as \emph{summing over the the bias} or
\emph{integrating over the fibre}.

\aaa[\label{s-cnc-cr}Closing remarks.]
%
%
We have just seen how biased dessins naturally arise in the study
unbiased dessins.  We give the last word to Alexander Grothendieck,
who seems to have anticipated this (see \cite{gr-edp}, p5 of AG's
manuscript):

\begin{quotation}
[L]e gens s'obstinent encore, en calculant avec des groups
fondamentaux, \`a fixer un seul point base, plut\^ot que d'en chosir
astucieusement tout un paquet qui soit invariant par les sym\'etries
de la situation [\ldots]
\end{quotation}

Or in English ~\cite{gr-sp}:
\begin{quotation}
[P]eople still obstinately persist, when calculating with fundamental
groups, in fixing a single base point, instead of cleverly choosing a
whole packet of points which is invariant under the symmetries of the
situation [\ldots]
\end{quotation}

\medskip
\rightline{Email: \texttt{jfine2358@gmail.com}}

\section*{Postscript (14 February 2018)}

Since this, the \emph{old paper}, the author wrote \emph{The algebra
  of balanced dessins} (arXiv:1802.04531).
The \emph{new paper} gives a key definition, for a new approach to
dessins and algebraic numbers. Its distant goal is to construct from
each dessin $D$ an algebraic number $\eta_D$, in a systematic and
useful way. The new paper defines the algebra of balanced
dessins. This algebra is generated by formal sums $\psi_D$ of
dessins. Each $\psi_D$ is intended to be intermediate between $D$ and
$\eta_D$.

The old paper contains the striking result, that every biased plane
tree determines several algebraic numbers. Unfortunately, these
algebraic numbers seem be obscure and unhelpful. We can't in practice
add or multiply them. They arise from analysis rather than algebra.

The old paper is also troubled by the problem of constructing, from a
suitable automorphim of the lattice of biased dessins, an automorphism
of the algebraic numbers. We seem to need something like this, to give
a combinatorial definition of the absolute Galois group.

Thinking on these difficulties, the author was led to the idea that it
would be very nice indeed if from \emph{any} dessin one could usefully
construct a single \emph{useful} algebraic number. The biased plane
tree result showed that much of the difficulty lay in multiplying the
algebraic numbers, that we wished to construct from the dessins.

Once one sees the utility of having a way to multiply the dessins
themselves, a potential way forward shows itself. Recall that the join
of two biased dessins is defined as an orbit in their Cartesian
product. The starting point of the new paper is precisely this product
on dessins.

The old paper was focussed on defining new invariants of dessins. This
was fairly widely felt to be important. However, if the program
inherent in the new approach succeeds, then every dessin has a very
nice invariant. It is the minimal polynomial of the algebraic number
constructed from that dessin.

The minimal polynomial of an algebraic number is, in some sense, its
universal and very best Galois invariant. Therefore, if the new
approach succeeds, then it also in some sense defines a universal
Galois invariant of dessins. (Computing this invariant is another
matter.)

The goals of the old paper were broad, and somewhat confused. Its `in
preparation' work might not appear. The new paper is the result of a
clarification and narrowing of goals.

\newpage

\vspace*{-1.3in}
\hspace*{-.75in}\includegraphics{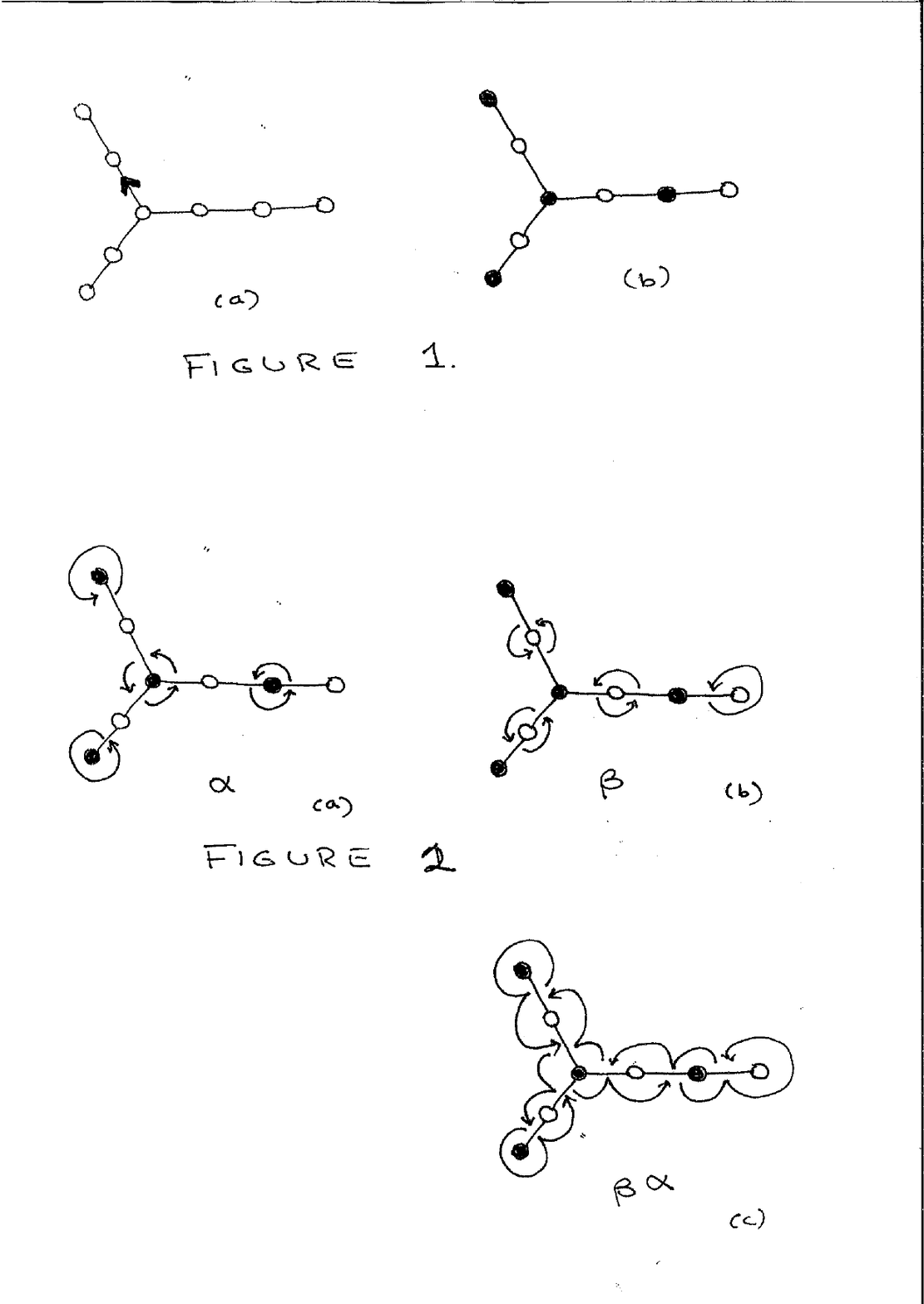}


\begin{thebibliography}{99}
\bibitem{jfine-dlbd-talk} Jonathan Fine, The decorated lattice of
  biased dessins (slides).
\bibitem{jfine-rpt-an} ------, Rooted plane trees and
  algebraic numbers (in preparation).
\bibitem{jfine-dlbd} ------, The decorated lattice of
  biased dessins (in preparation).
\bibitem{jfine-gibd} ------, Galois invariants of biased
  dessins (in preparation).
\bibitem{GGD} Girondo and Gonz\'alez-Diez, Introduction to Compact
  Riemann Surfaces and Dessins d'Enfants, Cambridge University Press
\bibitem{gr-edp} Grothendieck, Esquisse d'un programme, Geometric
  Galois Actions 1, LMS Lecture notes~242, Cambridge University Press
\bibitem{gr-sp} ------, Sketch of a program, Geometric Galois Actions
  1, LMS Lecture notes 242, Cambridge University Press
\bibitem{LZ} Lando and Zvonkin, Graphs on Surfaces and Their
  Applications, Springer Verlag
\bibitem{rstan-cn} Richard Stanley, Catalan Numbers, Cambridge
  University Press
\end{thebibliography}
\end{document}